\documentclass[12pt]{article}

\usepackage{latexsym,amsfonts,amsmath,amsthm,amssymb, epsfig}
\usepackage{color}

\usepackage[vcentering,dvips]{geometry}
\newtheorem{thm}{Theorem}
\newtheorem{lem}[thm]{Lemma}

\theoremstyle{remark}

\theoremstyle{definition}

\newcommand{\R}{\mathbb{R}}

\newcommand{\N}{\mathbb{N}}

\newcommand{\rd}{\,{\rm d}}
\newcommand{\bsx}{{\boldsymbol x}}
\newcommand{\bst}{{\boldsymbol t}}
\newcommand{\bszero}{{\boldsymbol 0}}

\newcommand{\cP}{\mathcal{P}}
\newcommand{\cA}{\mathcal{A}}

\title{The $L_p$-discrepancy for finite $p>1$ suffers from the curse of dimensionality}  

\author{Erich Novak and  Friedrich Pillichshammer}
\date{}

\begin{document}

\maketitle

\begin{abstract}
The $L_p$-discrepancy is a classical quantitative measure for the irregularity of distribution of an $N$-element point set in the $d$-dimensional unit cube. Its inverse for dimension $d$ and error threshold $\varepsilon \in (0,1)$ is the number of points in $[0,1)^d$ that is required such that the minimal normalized $L_p$-discrepancy is less or equal $\varepsilon$. It is well known, that the inverse of $L_2$-discrepancy grows exponentially fast with the dimension $d$, i.e., we have the curse of dimensionality, whereas the inverse of $L_{\infty}$-discrepancy depends exactly linearly on $d$. 
The behavior of inverse of $L_p$-discrepancy for general $p \not\in \{2,\infty\}$ was
an open problem since many years. 
Recently, the curse of dimensionality for the $L_p$-discrepancy was shown for an infinite sequence of values $p$ in $(1,2]$, but the general result seemed to be out of reach.

In the present paper we show that the $L_p$-discrepancy suffers from the curse 
of dimensionality for all $p$ in $(1,\infty)$ and only the case $p=1$ is still open.   

This result follows from a more general result that we show for the worst-case error 
of positive quadrature formulas for an anchored  Sobolev space
of once differentiable functions in each variable whose first 
mixed derivative has finite $L_q$-norm, where $q$ is the H\"older conjugate of~$p$.  
\end{abstract}

\centerline{\begin{minipage}[hc]{130mm}{
{\em Keywords:} Discrepancy, numerical integration, curse of dimensionality, tractability, quasi-Monte Carlo\\
{\em MSC 2010:} 11K38, 65C05, 65Y20}
\end{minipage}}

%E 
\section{Introduction and main result}\label{sec:intro}

For a set $\cP$ consisting of $N$ points $\bsx_1,\bsx_2,\ldots,\bsx_N$ in the $d$-dimensional unit-cube $[0,1)^d$ the local discrepancy function $\Delta_{\cP}:[0,1]^d \rightarrow \R$ is defined as $$\Delta_{\cP}(\bst)=\frac{|\{k \in \{1,2,\ldots,N\}\ : \ \bsx_k \in [\bszero,\bst)\}|}{N}-{\rm volume}([\bszero,\bst)),$$  
for $\bst=(t_1,t_2,\ldots,t_d)$ in $[0,1]^d$, where $[\bszero,\bst)=[0,t_1)\times [0,t_2)\times \ldots \times [0,t_d)$. For a parameter $p \in [1,\infty]$ the $L_p$-discrepancy of the point set $\cP$ is defined as the $L_p$-norm of the local discrepancy function $\Delta_{\cP}$, i.e., $$L_{p,N}(\cP):=\left(\int_{[0,1]^d} |\Delta_{\cP}(\bst)|^p \rd \bst\right)^{1/p} \quad \mbox{for $p \in [1,\infty)$,}$$ and $$L_{\infty,N}(\cP):=\sup_{\bst \in [0,1]^d} | \Delta_{\cP}(\bst)| \quad  \mbox{for $p=\infty$.}$$ Traditionally, the $L_{\infty}$-discrepancy is called star-discrepancy and is denoted by $D_N^{\ast}(\cP)$ rather than $L_{\infty,N}(\cP)$. The study of $L_p$-discrepancy has its roots in the theory of uniform distribution modulo one; see \cite{BC,DT97,kuinie,mat} for detailed information. 
It has a close relation to numerical integration, see Section~\ref{sec:int}.

Since one is interested in point sets with $L_p$-discrepancy as low as possible it is obvious to study for $d,N \in \N$ the quantity $${\rm disc}_p(N,d):=\min_{\cP} L_{p,N}(\cP),$$ where the minimum is extended over all $N$-element point sets $\cP$ in $[0,1)^d$. This quantity is called the $N$-th minimal $L_p$-discrepancy in dimension $d$.

Traditionally, the $L_p$-discrepancy is studied from the point of view of a fixed dimension $d$ and one asks for the asymptotic behavior for increasing sample sizes $N$. 
The celebrated result of Roth~\cite{Roth} is the most famous result %E  s deleted  
in this direction and can be seen as the initial point of discrepancy theory. For $p \in (1,\infty)$ it is known that for every dimension $d \in \N$ there exist positive reals $c_{d,p},C_{d,p}$ such that for every $N \ge 2$ it holds true that $$c_{d,p} \frac{(\log N)^{\frac{d-1}{2}}}{N} \le {\rm disc}_p(N,d) \le C_{d,p} \frac{(\log N)^{\frac{d-1}{2}}}{N}.$$ Similar results, but less accurate, are available also for $p\in \{1,\infty\}$. See the above references for further information. The currently best 
asymptotical lower bound in the $L_\infty$-case can be found in \cite{BLV}.  

All the classical bounds have a poor dependence on the dimension $d$. 
For large $d$ these bounds are only meaningful in an asymptotic sense 
(for very large $N$) and do not give any information about the 
discrepancy in the pre-asymptotic regime (see, e.g., \cite{NW10} or \cite[Section~1.7]{DKP} for discussions). Nowadays, motivated from applications of point sets with low discrepancy for numerical integration, there is dire need of information about the dependence of discrepancy on the dimension. 

This problem is studied with the help of the so-called inverse of $L_p$-discrepancy (or, in a more general context, the information complexity; see Section~\ref{sec:int}). This concept compares the minimal $L_p$-discrepancy with the initial discrepancy ${\rm disc}_p(0,d)$, which is the $L_p$-discrepancy of the empty point set, and asks for the minimal number $N$ of nodes that is necessary in order to achieve that the $N$-th minimal $L_p$-discrepancy is smaller than $\varepsilon$ times ${\rm disc}_p(0,d)$ for a threshold $\varepsilon \in (0,1)$. In other words, for $d \in \N$ and $\varepsilon \in (0,1)$ the inverse of the $N$-th minimal $L_p$-discrepancy is defined as $$N_p^{{\rm disc}}(\varepsilon,d):=\min\{N \in \N \ : \ {\rm disc}_p(N,d) \le \varepsilon \ {\rm disc}_p(0,d)\}.$$ The question is now how fast $N_p^{{\rm disc}}(\varepsilon,d)$ increases, when $d \rightarrow \infty$ and $\varepsilon \rightarrow 0$. 

It is well known and easy to check that for the initial $L_p$-discrepancy we have 
\begin{equation}\label{initdisc}
{\rm disc}_p(0,d)=\left\{ 
\begin{array}{ll}
\frac{1}{(p+1)^{d/p}} & \mbox{if $p \in [1,\infty)$},\\[0.5em]
1 & \mbox{if $p=\infty$}.
\end{array}
\right.
\end{equation}
Here we observe a difference in the cases of finite and infinite $p$. While for $p=\infty$ the initial discrepancy equals 1 for every dimension $d$, for finite values of $p$ the initial discrepancy tends to zero exponentially fast with the dimension. 
 
For $p \in \{2,\infty\}$ the behavior of the inverse of $N$-th minimal $L_p$-discrepancy is well understood. In the $L_2$-case it is known that for all $\varepsilon \in (0,1)$ we have $$(1.125)^d (1-\varepsilon^2) \le N_2^{{\rm disc}}(\varepsilon,d) \le 1.5^d \varepsilon^{-2}.$$ Here the lower bound was first shown by Wo\'{z}niakowski in \cite{Wo99} (see also \cite{NW01,NW10}) and the upper bound follows from an easy averaging argument, see, e.g., \cite[Sec.~9.3.2]{NW10}. 

In the $L_{\infty}$-case it was shown by Heinrich, Novak, Wasilkowski and Wo\'{z}nia\-kow\-ski in \cite{hnww} that there exists an absolute positive constant $C$ such that for every $d \in \N$ and $\varepsilon \in (0,1)$ we have $$N_{\infty}^{{\rm disc}}(\varepsilon,d) \le C d \varepsilon^{-2}.$$ The currently smallest known value of $C$ is $6.23401\ldots$ as shown in \cite{GPW}. On the other hand, Hinrichs~\cite{Hi04} proved that there exist numbers $c>0$ and $\varepsilon_0 \in (0,1)$ such that for all $d \in \N$ and all $\varepsilon \in (0,\varepsilon_0)$ we have $$N_{\infty}^{{\rm disc}}(\varepsilon,d) \ge c d \varepsilon^{-1}.$$ 

So while the inverse of $L_2$-discrepancy grows exponentially fast with the dimension $d$, 
the inverse of the star-discrepancy depends only linearly on the dimension $d$. 
One says that the $L_2$-discrepancy suffers from the curse of dimensionality. 
In information based complexity theory the behavior of the inverse 
of star-discrepancy is called ``polynomial tractability'' (see, e.g., \cite{NW10}).
Hence the situation is clear (and quite different) for $p \in \{2,\infty\}$. 
But what happens for all other $p \not \in \{2,\infty\}$? This question was open for many years. 

Quite recently we proved in \cite{NP23} that the $L_p$-discrepancy 
suffers from the curse for all values $p$ of the 
form $$p=\frac{2 \ell}{2 \ell-1} \quad \mbox{ with $\ell \in \mathbb{N}$}.$$
We used ideas that only work if $q$ (the H\"older conjugate of $p$) is an even integer and therefore could not handle other values of $p$. Now, with a  different approach,  we  solve the  question for all $p \in (1,\infty)$.

\begin{thm}\label{thm1}
For every $p$ in $(1,\infty)$ there exists a real $C_p$ that is strictly larger than 1, such that for all $d \in \N$ and all $\varepsilon \in (0,1/2)$ we have $$N_p^{{\rm disc}}(\varepsilon,d) \ge C_p^d \, (1-2 \varepsilon).$$ We have $$C_p=\left(\frac{1}{2}+\frac{p+1}{p}  \frac{1+2^{p/(p+1)}-2^{1/(p+1)}}{4}\right)^{-1} > 1.$$ 
In particular, for all $p$ in $(1,\infty)$ the $L_p$-discrepancy suffers from the curse of dimensionality.
\end{thm}

Figure~\ref{fig_ca} shows the graph of $C_p$ for $p \in [1,50]$. An improvement will be given in Section~\ref{sec:RemOQ}.

  \begin{figure}
  \begin{center}
  \includegraphics[width=10cm]{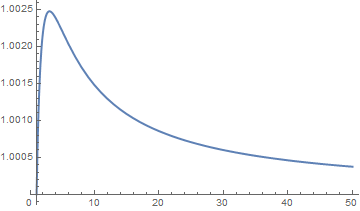}
  \caption{Plot of $C_p$ for $p \in [1,50]$. Note that $C_1=1$ and $\lim_{p \rightarrow \infty}
   C_p=1$. We have $C_2=1.0022\ldots$, $C_3=1.00248\ldots$, $C_4=1.00238\ldots$.}
  \label{fig_ca}
  \end{center}
  \end{figure}

The result will follow from a more general result about the integration problem in the anchored Sobolev space with a $q$-norm that will be introduced and discussed in the following Section~\ref{sec:int}. This result will be stated as Theorem~\ref{thm2}.
We end the paper with three open problems. 

\section{Relation to numerical integration}\label{sec:int}

It is well known that the $L_p$-discrepancy is related to multivariate integration (see, e.g., \cite[Chapter~9]{NW10}). Since this relation is essential for the present approach we repeat the brief summary from \cite[Section~2]{NP23}. From now on let $p,q \ge 1$ be H\"older conjugates, i.e., $1/p+1/q=1$. For $d=1$ let $W_q^1([0,1])$ be the space of absolutely continuous functions whose first derivatives belong to the space $L_q([0,1])$. For $d>1$ consider the $d$-fold tensor product space which is denoted by $$W_q^{\boldsymbol{1}}:=W_q^{(1,1,\ldots,1)}([0,1]^d)$$ and which is the Sobolev space of functions on $[0,1]^d$ that are once differentiable in each variable and whose first derivative $\partial^d f/\partial \bsx$ has finite $L_q$-norm, where $\partial \bsx=\partial x_1 \partial x_2 \ldots \partial x_d$. Now consider the subspace of functions that satisfy the boundary conditions $f(\bsx)=0$ if at least one component of $\bsx=(x_1,\ldots,x_d)$ equals 0 and equip this subspace with the norm $$\|f\|_{d,q}:=\left(\int_{[0,1]^d} \left|\frac{\partial^d}{\partial \bsx}f(\bsx)\right|^q \rd \bsx \right)^{1/q} \quad \mbox{for $q \in [1,\infty)$,}$$ and $$\|f\|_{d,\infty}:=\sup_{\bsx \in [0,1]^d}\left|\frac{\partial^d}{\partial \bsx}f(\bsx)\right|  \quad \mbox{for $q=\infty$.}$$ That is, consider the space $$F_{d,q}:=\{f \in W_q^{\boldsymbol{1}} \ : \ f(\bsx)=0\ \mbox{if $x_j=0$ for some $j \in [d]$ and $\|f\|_{d,q}< \infty$}\}.$$ 

Now consider multivariate integration $$I_d(f):=\int_{[0,1]^d} f(\bsx) \rd \bsx \quad \mbox{for $f \in F_{d,q}$}.$$
We approximate the integrals $I_d(f)$ by linear algorithms of the form 
\begin{equation}\label{def:linAlg}
A_{d,N}(f)=\sum_{k=1}^N a_j f(\bsx_k),
\end{equation}
where $\bsx_1,\bsx_2,\ldots,\bsx_N$ are in $[0,1)^d$ and $a_1,a_2,\ldots,a_N$ are real weights that we call integration weights. If $a_1=a_2=\ldots =a_N=1/N$, then the linear algorithm \eqref{def:linAlg} is a so-called quasi-Monte Carlo algorithm, and we denote this by $A_{d,N}^{{\rm QMC}}$.

Define the worst-case error of an algorithm \eqref{def:linAlg} by 
\begin{equation}\label{eq:wce}
e(F_{d,q},A_{d,N})=\sup_{f \in F_{d,q}\atop \|f\|_{d,q}\le 1} \left|I_d(f)-A_{d,N}(f)\right|.
\end{equation}
For a quasi-Monte Carlo algorithm $A_{d,N}^{{\rm QMC}}$ it is well known that %(see, e.g., \cite{SW1998}) that 
$$
e(F_{d,q},A_{d,N}^{{\rm QMC}})= L_{p,N}(\overline{\cP}),
$$ 
where $L_{p,N}(\overline{\cP})$ is the $L_p$-discrepancy of the point set
\begin{equation}\label{def:oP}
\overline{\cP}=\{\boldsymbol{1} - \bsx_k \ : \ k\in \{1,2,\ldots,N\}\},
\end{equation}
where $\boldsymbol{1} - \bsx_k$ is defined as the component-wise difference of the vector containing only ones and $\bsx_k$, 
see, e.g.,  \cite[Section~9.5.1]{NW10} for the case $p=2$.

For general linear algorithms  \eqref{def:linAlg} the worst-case error is the so-called generalized $L_p$-discrepancy 
$$
e(F_{d,q},A_{d,N})= \overline{L}_{p,N}(\overline{\cP},\cA),
$$ 
where $\overline{\cP}$ is like in \eqref{def:oP} and $\cA$ consists of exactly the coefficients from the given linear algorithm (see \cite{NW10}). 
Here for points $\cP=\{\bsx_1,\bsx_2,\ldots,\bsx_N\}$ and corresponding 
coefficients $\cA=\{a_1,a_2,\ldots,a_N\}$ the discrepancy 
function is $$\overline{\Delta}_{\cP,\cA}(\bst)=\sum_{k=1}^N a_k {\bf 1}_{[\boldsymbol{0},\bst)}(\bsx_k) 
- t_1 t_2\cdots t_d$$ for $\bst=(t_1,t_2,\ldots,t_d)$ in $[0,1]^d$ and the generalized $L_p$-discrepancy is $$\overline{L}_{p,N}(\cP,\cA)=\left(\int_{[0,1]^d} |\overline{\Delta}_{\cP,\cA}(\bst)|^p \rd \bst\right)^{1/p} \quad \mbox{for $p \in [1,\infty)$,}$$ with the usual adaptions for $p=\infty$. If $a_1=a_2=\ldots=a_N=1/N$, then we are back to the classical definition of $L_p$-discrepancy from Section~\ref{sec:intro}.

From this point of view we now study the more general problem of numerical integration in $F_{d,q}$ rather than only the $L_p$-discrepancy (which corresponds to quasi-Monte Carlo algorithms -- although with suitably ``reflected'' points). We consider linear algorithms where we restrict ourselves % exclusively 
to non-negative weights $a_1,\ldots,a_N$ (thus QMC-algorithms are included in our setting).

We define the $N$-th minimal worst-case error as $$e_q(N,d):=\min_{A_{d,N}} |e(F_{d,q},A_{d,N})|$$ where the minimum is extended over all linear algorithms of the form \eqref{def:linAlg} based on $N$ function evaluations along points $\bsx_1,\bsx_2,\ldots,\bsx_N$ from $[0,1)^d$ and with non-negative weights $a_1,\ldots,a_N \ge 0$. Note that for all $d,N \in \N$ we have 
\begin{equation}\label{ine:errdisc}
e_q(N,d) \le {\rm disc}_p(N,d).
\end{equation} 

The initial error is $$e_q(0,d)=\sup_{f \in F_{d,q}\atop \|f\|_{d,q}\le 1} \left|I_d(f)\right|.$$ 

We call $f \in F_{d,q}$ a worst-case function, if $I_d(f/\|f\|_{d,q})=e_q(0,d)$.

\begin{lem}\label{le:interr}
Let $d \in \N$ and let $q \in (1,\infty]$ and $p \in [1,\infty)$ with $1/p+1/q=1$. Then we have $$e_q(0,d)=\frac{1}{(p+1)^{d/p}}$$ and the worst-case function in $F_{d,q}$ is given by $h_d(\bsx)=h_1(x_1) \cdots h_1(x_d)$ for $\bsx=(x_1,\ldots,x_d) \in [0,1]^d$, where $h_1(x)=1-(1-x)^p$. Furthermore, we have $$\int_0^1 h_1(t)\rd t= \frac{p}{p+1}\quad \mbox{ and } \quad \|h_1\|_{1,q}= \frac{p}{(p+1)^{1/q}}.$$
\end{lem}

For a proof of Lemma~\ref{le:interr} see \cite{NP23}.

Note that for all H\"older conjugates $q \in (1,\infty]$ and $p \in [1,\infty)$ and for all $d \in \N$ we have $$e_q(0,d)={\rm disc}_p(0,d).$$

Now we define the information complexity as the minimal number of function evaluations necessary in order to reduce the initial error by a factor of $\varepsilon$. 
For $d \in \N$ and $\varepsilon \in (0,1)$ put 
$$N^{{\rm int}}_q(\varepsilon,d):= \min\{N \in \N \ : \ e_q(N,d) \le \varepsilon\ e_q(0,d)\}.$$
We stress that $N^{{\rm int}}_q(\varepsilon,d)$
is a kind of restricted complexity since we only allow positive quadrature formulas.

From \eqref{ine:errdisc}, \eqref{initdisc} and Lemma~\ref{le:interr} it follows that for all  H\"older conjugates $q \in (1,\infty]$ and $p \in [1,\infty)$ and for all $d \in \N$ and $\varepsilon \in (0,1)^d$ we have  $$N^{{\rm int}}_q(\varepsilon,d) \le N^{{\rm disc}}_p(\varepsilon,d).$$

Hence, Theorem~\ref{thm1} follows from the following more general result.

\begin{thm}\label{thm2}
For every $q$ in $(1,\infty)$ put $$C_p=\left(\frac{1}{2}+\frac{p+1}{p}  \frac{1+2^{p/(p+1)}-2^{1/(p+1)}}{4}\right)^{-1},$$ where $p$ is the H\"older conjugate of $q$. Then $C_p>1$ and for all $d \in \N$ and $\varepsilon \in (0,1/2)$ we have 
\begin{equation}\label{lbd:C}
N_q^{{\rm int}}(\varepsilon,d) \ge C_p^d \, (1-2 \varepsilon).
\end{equation}
In particular, for all $q$ in $(1,\infty)$ the integration problem in $F_{d,q}$ suffers from the curse of dimensionality for positive quadrature formulas. 
\end{thm}

\begin{proof}%[Proof of Theorem~\ref{thm2}] 
The proof of Theorem~\ref{thm2} is based on a suitable decomposition of the worst-case function $h_1$ from Lemma~\ref{le:interr}. This decomposition depends on $q$ and $p$, respectively, and will determine the value of $C_p$ in \eqref{lbd:C}.

For a decomposition point $a$ in $(0,1)$ that will be 
determined in a moment define the functions 
\begin{eqnarray*}
h_{1,1}(x) & = & \frac{1-(1-a)^p}{a} \min(x,a),\\
h_{1,2,(0)}(x) & = & \mathbf{1}_{[0,a]}(x)\left((1-(1-x)^p)-\frac{x}{a}(1-(1-a)^p) \right),\\
h_{1,2,(1)}(x) & = & \mathbf{1}_{[a,1]}(x)\left((1-(1-x)^p)-(1-(1-a)^p) \right).
\end{eqnarray*}
Then we have $h_1(x)=h_{1,1}(x)+h_{1,2,(0)}(x)+h_{1,2,(1)}(x)$. See Figure~\ref{fig_z} for an illustration.

\begin{figure}
  \begin{center}
  \includegraphics[width=8cm]{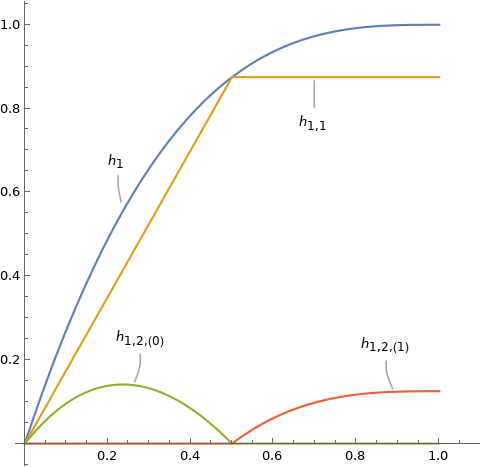}
  \caption{Decomposition functions $h_{1,1}$, $h_{1,2,(0)}$ and $h_{1,2,(1)}$ for $p=3$ and decomposition point $a=1/2$.}
  \label{fig_z}
  \end{center}
  \end{figure}

We choose the decomposition point $a$ such that  
$$
\Vert h_{1,1} + h_{1,2,(0)} \Vert_{1,q} = \Vert h_{1,1} + h_{1,2,(1)} \Vert_{1,q} =: \alpha .
$$
We have $$\Vert h_{1,1} + h_{1,2,(0)} \Vert_{1,q}^q=\int_0^a |p(1-x)^{p-1}|^q \rd x= p^q \frac{1-(1-a)^{p+1}}{p+1}$$ and $$\Vert h_{1,1} + h_{1,2,(1)} \Vert_{1,q}^q=\int_a^1 |p(1-x)^{p-1}|^q \rd x= p^q \frac{(1-a)^{p+1}}{p+1}.$$ Thus we have to choose $$a=1-\frac{1}{2^{1/(p+1)}}$$ and then $$\alpha=\frac{1}{2^{1/q}} \frac{p}{(p+1)^{1/q}} =\frac{\|h_1\|_{1,q}}{2^{1/q}}.$$ Obviously, $\alpha < \Vert h_1 \Vert_{1,q}$. 
Put 
$$
\beta := \max \left( \int_0^1 h_{1,1}(x) + h_{1,2,(0)}(x)\rd x , \int_0^1 h_{1,1}(x) + h_{1,2,(1)}(x)\rd x \right) 
$$
Then we have $\beta < \int_0^1 h_1(x)\rd x=\frac{p}{p+1}$.  

\bigskip

Consider a linear algorithm $A_{d,N}$ based in nodes $\bsx_1,\ldots,\bsx_N$ 
and non-negative weights $a_1,\ldots,a_N\ge 0$.  
Then for $i \in \{1,\ldots,N\}$ we define functions 
$$
P_i(\bsx) = \prod_{k=1}^d (h_{1,1}(x_k) + h_{1,2,(z_k)}(x_k)),\quad \bsx=(x_1,\ldots,x_d)\in [0,1]^d,  
$$
where $z_k \in \{ 0,1 \}$, $k\in \{1,\ldots,d\}$ are chosen in a way, such that 
$$
P_i (\bsx_i) = h_d (\bsx_i) .
$$
The functions $P_i$ are tensor products and therefore we have 
the simple formulas 
$$\|P_i\|_{d,q} =\alpha^d \quad \mbox{ and } \quad \int_{[0,1]^d} P_i(\bsx) \rd \bsx \le \beta^d$$
for every $i \in \{1,\ldots,N\}$. 

In order to estimate the error of $A_{d,N}$ we consider the two functions $h_d$ and 
$$
f^*= \sum_{i=1}^N P_i.
$$
For every $i \in \{1,\ldots,N\}$ we have 
$$
h_d (\bsx_i) \le f^* (\bsx_i) ,
$$
because $P_k \ge 0$ for all $k$ and $P_i (\bsx_i) = h_d (\bsx_i)$. 
This and the use of   %  exclusively 
non-negative quadrature weights implies that  
$$
A_{d,N} (h_d) \le A_{d,N}(f^*). 
$$
Now for real $y$ we use the notation $(y)_+:= \max (y, 0)$.  Then we have the error estimate 
\begin{equation}\label{errest1}
e(F_{d,q},A_{d,N})  \ge \frac{ \left(\int_{[0,1]^d} h_d(\bsx)\rd \bsx - \int_{[0,1]^d} f^*(\bsx)\rd \bsx\right)_+}{2 \max ( \Vert h_d \Vert_{d,q}, \Vert f^* \Vert_{d,q})},
\end{equation}
which is trivially true if $\int_{[0,1]^d} h_d(\bsx)\rd \bsx \le  \int_{[0,1]^d} f^*(\bsx)\rd \bsx$ and which is easily shown if $\int_{[0,1]^d} h_d(\bsx)\rd \bsx >  \int_{[0,1]^d} f^*(\bsx)\rd \bsx$, because then 
\begin{eqnarray*}
\lefteqn{\left(\int_{[0,1]^d} h_d(\bsx)\rd \bsx - \int_{[0,1]^d} f^*(\bsx)\rd \bsx\right)_+}\\
& \le & \int_{[0,1]^d} h_d(\bsx)\rd \bsx - A_{d,N}(h_d)+A_{d,N}(f^*)-\int_{[0,1]^d} f^*(\bsx)\rd \bsx\\
& \le & \|h_d\|_{d,q} \, e(F_{d,q},A_{d,N})+\|f^*\|_{d,q} \, e(F_{d,q},A_{d,N})\\
& \le & 2 \max(\|h_d\|_{d,q},\|f^*\|_{d,q}) \, e(F_{d,q},A_{d,N}).
\end{eqnarray*}

From the triangle inequality we obtain 
$$
\Vert f^* \Vert_{d,q} \le N \alpha^d=N \frac{\|h_d\|_{d,q}}{2^{d/q}},
$$ 
and we also have 
$$
\int_{[0,1]^d} f^*(\bsx)\rd \bsx \le N  \beta^d.
$$ 
Inserting into \eqref{errest1} yields 
$$
e(F_{d,q},A_{d,N}) \ge \frac{ \left(\int_{[0,1]^d} h_d(\bsx)\rd \bsx - N \beta^d\right)_+} {2 \Vert h_d \Vert_{d,q} \max (1, 
N/2^{d/q} ) } .
$$
Since the right hand side is independent of $\bsx_1,\ldots,\bsx_N$ and $a_1,\ldots,a_N$ we obtain 
\begin{equation}\label{lberr}
e_q(N,d)\ge \frac{ \left(\int_{[0,1]^d} h_d(\bsx)\rd \bsx - N \beta^d\right)_+} {2 \Vert h_d \Vert_{d,q} \max (1, 
N/2^{d/q} )  } .
\end{equation}
Put $\gamma:=\frac{1}{\beta} \int_0^1 h_1(x)\rd x=\frac{1}{\beta} \frac{p}{p+1}$. Then we have $\gamma >1$.

Now let $\varepsilon \in (0,1/2)$ and assume that $e_q(N,d) \le \varepsilon \, e_q(0,d)$. This and $e_q(0,d)\|h_d\|_{d,q}=(\frac{p}{p+1})^d$ implies that
$$2 \varepsilon  \left(\frac{p}{p+1}\right)^d \max \left(1, 
\frac{N}{2^{d/q}}\right)  \ge  \left(\int_{[0,1]^d} h_d(\bsx)\rd \bsx - N \beta^d\right)_+.$$
If $N \le \min(2^{d/q},\gamma^d)$, then we obtain
\begin{align*}
\left(\frac{p}{p+1}\right)^d  - N \beta^d  = & \int_{[0,1]^d} h_d(\bsx)\rd \bsx - N \beta^d \\
= &  \left(\int_{[0,1]^d} h_d(\bsx)\rd \bsx - N \beta^d\right)_+ \\
\le & \, 2 \varepsilon \left(\frac{p}{p+1}\right)^d \max \left(1, \frac{N}{2^{d/q}}\right)\\
= & \, 2 \varepsilon \left(\frac{p}{p+1}\right)^d. 
\end{align*}
Hence $$N \ge \left(\frac{1}{\beta} \frac{p}{p+1}\right)^d (1-2 \varepsilon) \ge (\min(2^{1/q},\gamma))^d (1-2 \varepsilon).$$
If $N \ge \min(2^{d/q},\gamma^d)$, then we trivially have $$N \ge (\min(2^{1/q},\gamma))^d (1-2 \varepsilon).$$ 
This yields $$N_q^{{\rm int}}(\varepsilon,d)\ge (\min(2^{1/q},\gamma))^d (1-2 \varepsilon),$$ and we are done.

It remains to compute the values for $C_p:=\min(2^{1/q},\gamma)$. Obviously, $C_p>1$, since $\gamma>1$. For $a=1-\frac{1}{2^{1/(p+1)}}$ we have 
\begin{align*}
\int_0^1 h_{1,1}(x)+h_{1,2,(0)}(x) \rd x = & \int_0^a 1-(1-x)^p \rd x + \int_a^11-(1-a)^p \rd x\\
 = & \, a-\frac{1-(1-a)^{p+1}}{p+1}+(1-a) (1-(1-a)^p)\\
% = & 1-\frac{1}{2^{1/(p+1)}}-\frac{1}{2(p+1)}+\frac{1}{2^{1/(p+1)}} \left(1-\frac{1}{2^{p/(p+1)}}\right) \\
  = & \, \frac{1}{2} \frac{p}{p+1}
\end{align*}
and
\begin{align*}
\int_0^1 h_{1,1}(x)+h_{1,2,(1)}(x) \rd x = & \int_0^a \frac{1-(1-a)^p}{a} x \rd x + \int_a^1 1-(1-x)^p \rd x\\
 = & \,  \frac{a (1-(1-a)^p)}{2} +1-a-\frac{(1-a)^{p+1}}{p+1}\\
% = & \frac{1}{2} \left(1-\frac{1}{2^{1/(p+1)}}\right)\left(1-\frac{1}{2^{p/(p+1)}}\right)+\frac{1}{2^{1/(p+1)}} -\frac{1}{2 (p+1)}\\
 = & \,  \frac{1}{2} \frac{p}{p+1} + \frac{1+2^{p/(p+1)}-2^{1/(p+1)}}{4}.
\end{align*}
Hence $$\beta=\frac{1}{2} \frac{p}{p+1} + \frac{1+2^{p/(p+1)}-2^{1/(p+1)}}{4}.$$ Therefore $$\gamma=\frac{1}{\frac{1}{2}+\frac{p+1}{p}  \frac{1+2^{p/(p+1)}-2^{1/(p+1)}}{4}}$$ 
and $$C_p=\min(2^{1/q},\gamma)= \min(2^{1-1/p},\gamma)=\gamma=\frac{1}{\frac{1}{2}+\frac{p+1}{p}
  \frac{1+2^{p/(p+1)}-2^{1/(p+1)}}{4}}.$$
\end{proof}

\section{Remark and open problems}\label{sec:RemOQ}   

There are many spaces where the curse of dimensionality 
is present for high dimensional integration, 
see the recent survey paper \cite{No24}. 
For $p=\infty$ the discrepancy is tractable and 
this property does not hold for many unweighted 
spaces; for a recent example see Krieg~\cite{Kr24}.

\bigskip

The value of $C_p$ in Theorem~\ref{thm1} and \ref{thm2} (see also Figure~\ref{fig_ca}) can be improved with the following ``spline'' method:

For $y \in [0,1]$ define the linear splines
$$s_y(x)=\left\{ 
\begin{array}{ll}
\frac{1-(1-y)^p}{y} x & \mbox{if $x < y$,}\\[0.5em]
1-(1-y)^p & \mbox{if $x\ge y$.}
\end{array}
\right.
$$
Then $s_y \in F_{1,q}$, $s_y \ge 0$ and $s_y(y)=h_1(y)$, the worst-case function from Lemma~\ref{le:interr}. Put \begin{equation*}%\label{est:alphabeta}
\alpha:= \max_{y \in [0,1]} \|s_y\|_{1,q} \quad \mbox{and}\quad \beta:= \max_{y \in [0,1]} \int_0^1 s_y(x)\rd x.
\end{equation*} 
It is easily shown that $$\alpha=\max_{y \in [0,1]} \frac{1-(1-y)^p}{y^{1/p}}\quad \mbox{and}\quad \beta=\max_{y \in [0,1]} (1-(1-y)^p)\left(1-\frac{y}{2}\right).$$ Furthermore, it is elementary but tedious to show that $\alpha < \|h_1\|_{1,q}$. On the other hand, since for every $y \in [0,1]$ we have $0 \le s_y(x) < h_1(x)$, it is clear that $\beta < \int_0^1 h_1(x) \rd x$. This yields that
\begin{equation}\label{def:C1new}
\widetilde{C}_p:=\min\left(\frac{\|h_1\|_{1,q}}{\alpha},\frac{1}{\beta} \int_0^1 h_1(x)\rd x\right) >1.
\end{equation}

Now we modify the approach in the proof of Theorem~\ref{thm2} in the following way: As before, consider a linear algorithm $A_{d,N}$ based in nodes $\bsx_1,\ldots,\bsx_N$ in $[0,1]^d$ and with non-negative weights $a_1,\ldots,a_N\ge 0$.  Let $x_{i,j}$ be the $j$-th coordinate of the point $\bsx_i$, $i \in \{1,\ldots,N\}$ and $j \in \{1,\ldots,d\}$. For $i \in \{1,\ldots,N\}$ we define functions 
$$
P_i(\bsx) := s_{x_{i,1}}(x_1) s_{x_{i,2}}(x_2) \cdots s_{x_{i,d}}(x_d),\quad \bsx=(x_1,\ldots,x_d)\in [0,1]^d.  
$$

Consider the two functions $h_d$ and $f^*:= \sum_{i=1}^N P_i$. Since $A_{d,N}$ uses only non-negative weights we have 
\begin{align*}
A_{d,N}(f^*) = & \sum_{i=1}^N a_i \sum_{j=1}^N P_j(\bsx_i) \ge \sum_{i=1}^N a_i P_i(\bsx_i)\\
= & \sum_{i=1}^N a_i s_{x_{i,1}}(x_{i,1}) s_{x_{i,2}}(x_{i,2}) \cdots s_{x_{i,d}}(x_{i,d}) =\sum_{i=1}^N a_i h_d(\bsx_i) = A_{d,N}(h_d).
\end{align*}
In the same way as in the proof of Theorem~\ref{thm2} we obtain  
\begin{align*}
e(F_{d,q},A_{d,N})  \ge & \frac{ \left(\int_{[0,1]^d} h_d(\bsx)\rd \bsx - \int_{[0,1]^d} f^*(\bsx)\rd \bsx\right)_+}{2 \max ( \Vert h_d \Vert_{d,q}, \Vert f^* \Vert_{d,q})}\\
\ge & \frac{\left(\int_{[0,1]^d} h_d(\bsx)\rd \bsx - N \beta^d\right)_+} {2 \max(\|h_d\|_{d,q},N \alpha^d)},
\end{align*}
where we used that $\Vert f^* \Vert_{d,q} \le N \alpha^d$ and $\int_{[0,1]^d} f^*(\bsx)\rd \bsx \le N  \beta^d$. This yields  
\begin{equation*}
e_q(N,d)\ge \frac{\left(\int_{[0,1]^d} h_d(\bsx)\rd \bsx - N \beta^d\right)_+} {2 \max(\|h_d\|_{d,q},N \alpha^d)}.
\end{equation*}
From here it follows in the same way as in the proof of Theorem~\ref{thm2} that $$N_q^{{\rm int}}(\varepsilon,d)\ge \widetilde{C}_p^d \, (1-2 \varepsilon) \quad \mbox{where $\widetilde{C}_p>1$ is given in \eqref{def:C1new}.}$$ 

This re-proves Theorem~\ref{thm2} (and Theorem~\ref{thm1}). The advantage of $C_p$ in Theorem~\ref{thm1} and \ref{thm2} is that it is stated explicitly for any $p \in (1,\infty)$. The value of $\widetilde{C}_p$ can be computed numerically for every $p \in (1,\infty)$. Experiments show a strong improvement of $\widetilde{C}_p$ over $C_p$. See the following table and Figure~\ref{fig_caNEU}:

$$
\begin{array}{l||l|l|l|l}
p & 2 & 3 & 4 & 5 \\
\hline 
C_p  & 1.0022\ldots & 1.00248\ldots & 1.00238\ldots & 1.0022\ldots\\
\widetilde{C}_p  &  1.06066\ldots & 1.07231 \ldots & 1.07276\ldots & 1.07005\ldots
\end{array}
$$

$$
\begin{array}{l||l|l|l|l}
p & 10 & 20 & 30 & 100 \\
\hline 
C_p & 1.00148\ldots & 1.00086\ldots & 1.0006\ldots & 1.00019\ldots\\
\widetilde{C}_p  &  1.05327\ldots & 1.035 \ldots & 1.02627\ldots & 1.0101\ldots
\end{array}
$$

Figure~\ref{fig_caNEU} shows the strong improvement of $\widetilde{C}_p$ over $C_p$. The blue line is the graph from Figure~\ref{fig_ca}. 

\begin{figure}[h]
  \begin{center}
  \includegraphics[width=10cm]{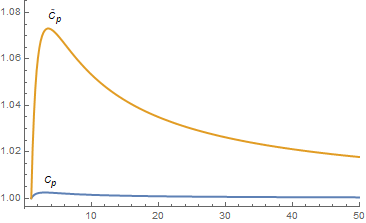}
  \caption{Plot of $\widetilde{C}_p$ compared to $C_p$ for $p \in [1,50]$. Note that $\widetilde{C}_1=C_1=1$.}
  \label{fig_caNEU}
  \end{center}
  \end{figure}

\bigskip

We end the paper with three open problems. 

\begin{enumerate} 

\item
In order to estimate the error of quadrature formulas, we only 
considered two functions $h_d$ and 
$
f^*= \sum_{i=1}^N P_i . 
$
The reason is that we have an exact formula for the norm of $P_i$, 
while the norms of ``better'' fooling functions are difficult to estimate. 
Our first Open Problem is to improve the lower bounds by finding bigger values 
for the constant $C_p$ in the main result and $\widetilde{C}_p$. 

\item 
The proof in \cite{NP23} only works for even $q$. 
However, for even $q$, it is more general since we 
prove the lower bound for all quadrature formulas, the weights $a_i$ 
do not have to be positive. 
Hence we ask, this is Open Problem 2,  whether the curse also holds 
for all $p \in (1,\infty)$ for arbitrary quadrature formulas. 
We guess that the answer is yes, but our attempts to prove it  failed. 

\item 
We already mentioned that the problem is still open for $p=1$. 
Our technique does not work in this case and we even do not guess 
an answer to  this third Open Problem.  

\end{enumerate}

\vspace{0.5cm}
\noindent{\bf Author's Address:}\\

\noindent Erich Novak, Mathematisches Institut, FSU Jena, Ernst-Abbe-Platz 2, 07740 Jena, Germany. Email: erich.novak@uni-jena.de\\

\noindent Friedrich Pillichshammer, Institut f\"{u}r Finanzmathematik und Angewandte Zahlentheorie, JKU Linz, Altenbergerstra{\ss}e 69, A-4040 Linz, Austria. Email: friedrich.pillichshammer@jku.at

\end{document}